\documentclass[a4paper,10pt]{amsart}
\usepackage{amsmath,amsthm,xypic,
amssymb,amsfonts,
pdfsync,stmaryrd,mathrsfs,yfonts
}
\newtheorem{thm}{Theorem}[section]

\newtheorem{prop}[thm]{Proposition}
\newtheorem{cor}[thm]{Corollary}

\newtheorem{lem}[thm]{Lemma}
\theoremstyle{definition}

\newtheorem{defi}[thm]{Definition}

\newtheorem{remark}[thm]{Remark}
\newtheorem{rem}[thm]{Remark}
\newtheorem{rems}[thm]{Remarks}

\newtheorem{notation}[thm]{Notation}

\theoremstyle{plain}

\newcommand{\lla}{\langle\!\langle}
\newcommand{\rra}{\rangle\!\rangle}

\renewcommand{\phi}{\varphi}

\newcommand{\alt}{\mathop{\mathrm{Alt}}}

\newcommand{\sym}{\mathop{\mathrm{Sym}}}

\newcommand{\nrd}{\mathop{\mathrm{Nrd}}}

\newcommand{\car}{\mathop{\mathrm{char}}}
\newcommand{\id}{\mathop{\mathrm{id}}}

\newcommand{\disc}{\mathop{\mathrm{disc}}}

\author{M. G. Mahmoudi, \ A.-H. Nokhodkar}
\title[Totally decomposable algebras with involution in characteristic $2$]
{
On totally decomposable
algebras with involution
in characteristic two
}

\begin{document}
\maketitle

\begin{abstract}
A necessary and sufficient condition for a central simple algebra with involution over a field of characteristic two to be decomposable as a tensor product of quaternion algebras with involution, in terms of its Frobenius subalgebras, is given.
It is also proved that a bilinear Pfister form, recently introduced by A. Dolphin, can classify totally decomposable central simple algebras of orthogonal type.
\\

\noindent
\emph{Mathematics Subject Classification:} 16W10, 16W25, 16K20, 11E39. \\
\end{abstract}

\section{Introduction}

An old result due to A. A. Albert states that every central simple algebra $A$
of degree $4$ which carries an involution of the first kind can be
decomposed as a tensor product of two quaternion algebras (see \cite[\S16]{knus}).
This result is no longer valid if $A$ is of degree $8$ by the examples
given in \cite{art} over fields of characteristic different from
$2$ and in \cite{rowen1984} over fields of
characteristic $2$.
In \cite{art}, it was also shown that if $A$ is of degree $2^n$ over a field of characteristic different from $2$, then $A$ decomposes into a tensor product of quaternion algebras if and only if there exists a finite square-central subset of $A$ (called a {\it q-generating set}) which satisfies some commuting properties.
Over a field of particular cohomological dimension, it is known that central simple algebras which carry an involution of the first kind can be decomposed as a tensor product of quaternion algebras (see \cite{kahn}, and \cite{barry-chapman} for a characteristic $2$ counterpart).
In \cite{barry2014}, a similar result was proved provided that the base field is of the $u$-invariant $\leqslant8$.

A closely related problem is to determine the conditions under which a
central simple algebra with involution $(A,\sigma)$ is {\it totally decomposable} (i.e., $(A,\sigma)$
decomposes
as a tensor product of $\sigma$-invariant quaternion algebras).
In \cite{rowen78}, it was shown that if $A$ is of degree $4$ over a
field of characteristic different from $2$ and
$\sigma$ is of symplectic type, then $A$ can be decomposed as a tensor
product of two $\sigma$-invariant quaternion algebras.
A proof of this result in characteristic $2$ was given in
\cite{streb}, also a characteristic independent proof of this result
and a criterion for decomposability in the case where $\sigma$ is
orthogonal can be found in \cite{pari}.
A similar criterion for the unitary case of degree $4$ and of arbitrary
characteristic was derived in \cite{kq}.
A cohomological invariant to detect decomposability for degree $8$
algebras with symplectic involution over a field of characteristic
different from $2$ can be found in \cite{gpt}.
For the case of degree $8$ algebras with orthogonal involution $(A,\sigma)$ over a field of
characteristic different from $2$, a criterion for
decomposability in terms of the Clifford algebra of $(A,\sigma)$ can be
found in \cite[(42.11)]{knus}, see also \cite[(3.10)]{tignol2010}.
For the case where $(A,\sigma)$ is split and of arbitrary degree $2^n$
over a field of characteristic different from $2$,
a decomposability criterion in term of higher degree invariants of a
quadratic form $q$, to which $\sigma$ is adjoint, can be found in
\cite{tignol2010}.

Another relevant problem is to find invariants which classify central
simple algebras with involution $(A,\sigma)$ up to conjugation.
Orthogonal involutions of degree $\leqslant4$
can be classified by their Clifford algebras \cite[\S15]{knus}, \cite[\S2]{lewis-tignol}.
A degree $4$ central simple algebra with symplectic involution $(A,\sigma)$ can be classified by a $3$-fold Pfister form or an Albert form associated to $\sigma$, see \cite{klst} and \cite[\S16]{knus}.

In this work we study the problems of decomposition and classification of central simple algebras with involution in the case of characteristic $2$.
In (\ref{der}), we show that a central simple algebra with involution
over a field of characteristic $2$ is totally decomposable if and only if there exists a symmetric and self-centralizing subalgebra $S=\Phi(A,\sigma)$ of $A$ such that $(i)$ $x^2\in F$ for every $x\in S$ and
$(ii)$ $\dim_FS=2^{r_F(S)}$, where $r_F(S)$ is the minimum rank of
$S$.
In the case where $(A,\sigma)$ is totally decomposable central simple algebra with
involution of orthogonal type we show that the aforementioned subalgebra $\Phi(A,\sigma)$, is
unique up to isomorphism (see (\ref{field})).
We prove the existence of a natural associative bilinear form
$\mathfrak{s}$ on $\Phi(A,\sigma)$, isometric to
a recently introduced bilinear Pfister form $\mathfrak{Pf}(A,\sigma)$ in
\cite{dolphin3}, thus providing a more intrinsic definition of
$\mathfrak{Pf}(A,\sigma)$
(see (\ref{lp}), (\ref{gen})).
In \cite[(7.5)]{dolphin3}, it was shown that for every splitting field $K$ of $A$, the involution $\sigma_K$ on $A_K$ is adjoint to the bilinear form $\mathfrak{Pf}(A,\sigma)_K$ (see \cite[(7.5)]{dolphin3}, compare \cite[(5.1)]{queguiner-tignol}), and it was asked (see \cite[(7.4)]{dolphin3}) if $\mathfrak{Pf}(A,\sigma)$ classify $(A,\sigma)$ up to conjugation.
Using the methods developed in the current work, we give in (\ref{phi}) an affirmative answer to this question.

\section{Preliminaries}
Let $V$ be a finite dimensional vector space over a field $F$ and let ${\mathfrak b}:V\times V\rightarrow F$ be a bilinear form.
Let $D_F(\mathfrak{b})=\{\mathfrak{b}(v,v):v\in V ~{\rm and}~\mathfrak{b}(v,v)\neq0\}$.
If $K/F$ is a field extension, the {\it extension} of $\mathfrak {b}$ to $V_K=V\otimes_F K$ is denoted by $\mathfrak {b}_K$.

The {\it orthogonal sum} and the {\it tensor product} of two bilinear forms ${\mathfrak b}_1$ and ${\mathfrak b}_2$ are denoted by
${\mathfrak b}_1\perp {\mathfrak b}_2$ and ${\mathfrak b}_1\otimes {\mathfrak b}_2$ respectively.
For $\alpha$ in $F^\times$, the group of invertible elements of $F$, we use the notation $\langle \alpha\rangle$ for the isometry class of the one-dimensional bilinear space $(V,{\mathfrak b})$ over $F$ defined by ${\mathfrak b}(u,v)=\alpha uv$.
The bilinear form $\perp_{i=1}^n\!\!\langle\alpha_i\rangle$ is denoted by $\langle \alpha_1,\cdots,\alpha_n\rangle$.

Let $F$ be a field and let $\alpha_1,\cdots,\alpha_n\in F^\times$.
The $2^n$-dimensional bilinear form $\langle 1,\alpha_1\rangle\otimes\cdots\otimes\langle 1,\alpha_n\rangle$
over $F$ is called a {\it bilinear $n$-fold Pfister form} and is denoted by $\lla\alpha_1,\cdots,\alpha_n\rra$.
If ${\mathfrak b}$ is a bilinear Pfister form then there exists a bilinear form ${\mathfrak b}'$, uniquely determined up to isometry, such that ${\mathfrak b}=\langle1\rangle\perp {\mathfrak b}'$ (see \cite[p. 16]{arason}).
The form ${\mathfrak b}'$ is called the {\it pure subform} of ${\mathfrak b}$.

A {\it quadratic form} over $F$ is a map $q:V\rightarrow F$ such that: $(1)$ $q(\alpha v)=\alpha^2q(v)$ for every $\alpha\in F$ and $v\in V$;
$(2)$ the map $\mathfrak{b}_q:V\times V\rightarrow F$ defined by $\mathfrak{b}_q(u,v)=q(u+v)-q(u)-q(v)$ for every $u,v\in V$ is a bilinear form.
We say that $q$ is {\it totally singular} if $\mathfrak{b}_q(u,v)=0$ for every $u,v\in V$.
For $\alpha_1,\cdots,\alpha_n\in F$, the isometry class of the $n$-dimensional totally singular quadratic form $(V,q)$ over $F$ defined by $q(v_1,\dots,v_n)=\alpha_1v_1^2+\cdots+\alpha_nv_n^2$ is denoted by $[\alpha_1]\perp\cdots\perp[\alpha_n]$.
The Clifford algebra of a quadratic form $(V,q)$ is denoted by $C(V)$.
We refer the reader to \cite[Ch. II]{elman} for basic definitions and facts regarding Clifford algebras and quadratic and bilinear forms in arbitrary characteristic.

Let $R$ be a ring.
An additive map $\delta:R\rightarrow R$ is called a {\it derivation}, if
$\delta(ab)=a\delta(b)+\delta(a)b$ for every $a,b\in R$.
For $a\in R$, the map $\delta_a:R\rightarrow R$ defined by $\delta_a(x)=ax-xa$ is a derivation of $R$ which is called the {\it inner derivation} induced by $a$.

All $F$-algebras considered in this work are supposed to be unital and associative.
The reader is referred to \cite[Ch. 12]{pierce} for basic notions concerning central simple algebras. We just recall that the {\it degree} of a central simple algebra $A$ is defined by $\deg_FA=\sqrt{\dim_FA}$.
Also for a subalgebra $B$ of $A$ the centralizer of $B$ in $A$ and the center of $B$ are denoted respectively by $C_A(B)$  and $Z(A)$.

A finite dimensional algebra $A$ over $F$ is called a
{\it Frobenius algebra} if $A$ contains a hyperplane $H$ that contains no nonzero left ideal of $A$;
alternatively $A$ is called a Frobenius $F$-algebra if there exists a nondegenerate bilinear form
$\mathfrak{b}:A\times A\rightarrow F$ which is associative, in the sense that $\mathfrak{b}(x,yz)=\mathfrak{b}(xy,z)$
for every $x,y,z\in A$.
In this work, we use the following fundamental result about the properties of Frobenius subalgebras of central simple algebras:
\begin{thm}{\rm\cite[(2.2.3)]{jacob}}\label{223}
Let $A$ be a central simple algebra over a field $F$ and let $S$ be a commutative Frobenius subalgebra of $A$ such that $\dim_FS=\deg_FA$.
\begin{itemize}
\item[$(i)$]
We have $C_A(S)=S$.
\item[$(ii)$] Every derivation of $S$ into $A$ can
be extended to an inner derivation of $A$.
\end{itemize}
\end{thm}
For further properties of Frobenius algebras see \cite{jacob} .

The {\it minimum rank} of a finite dimensional $F$-algebra $A$ which is denoted by $r_F(A)$ is the minimum number $r$
such that $A$ can be generated as an $F$-algebra by $r$ elements.
Also the {\it Loewy length} of $A$ which is denoted by $\ell\ell(A)$ is defined as the smallest positive integer $l$ such that $J(A)^l=0$; in other words $\ell\ell(A)$ is the nilpotency index of the Jacobson radical of $A$.

Let $A$ be a central simple algebra over a field $F$.
An {\it involution} on $A$ is an anti-automorphism $\sigma$ of $A$ such that $\sigma^2=\id$.
The involution $\sigma$ is called of {\it the first kind} if $\sigma|_F=\id$.
The set of {\it alternating} and {\it symmetric} elements of $(A,\sigma)$ are defined as follows:
\[\alt(A,\sigma)=\{a-\sigma(a):a\in A\},\quad \sym(A,\sigma)=\{a\in A:\sigma(a)=a\}.\]

An involution $\sigma$ of the first kind is said to be of {\it symplectic} type if over a splitting field of $A$, $\sigma$ becomes adjoint to an alternating bilinear form.
Otherwise, $\sigma$ is said to be of {\it orthogonal} type.
If $\car F=2$ and $\sigma$ is of the first kind, then it can be shown that $\sigma$ is of orthogonal type if and only if $1\notin\alt(A,\sigma)$, see \cite[(2.6)]{knus}.
The {\it discriminant} of an involution $\sigma$ of orthogonal type is denoted by $\disc\sigma$, see \cite[(7.1)]{knus}.

Let $F$ be a field of characteristic $2$.
A {\it quaternion algebra} over $F$ is a central simple $F$-algebra of degree $2$.
As an $F$-algebra, every quaternion algebra is generated by two elements $u$ and $v$ subject to the relations
\begin{align*}
u^2+u\in F,\quad v^2\in F^\times\quad {\rm and}\quad uv+vu=v.
\end{align*}
Furthermore $\{1,u,v,uv\}$ is a basis of $Q$ over $F$.

\section{Totally singular conic Frobenius algebras}
\begin{defi}
In analogy with \cite{gar}, we call an algebra $R$ over a field $F$ a {\it totally singular conic algebra} if $x^2\in F$ for every $x\in R$.
\end{defi}

\begin{rem}\label{local}
Let $F$ be a field of characteristic $2$ and let $R$ be a finite dimensional totally singular conic $F$-algebra.
It follows immediately that

$(i)$ $R$ is a local commutative algebra and its unique maximal ideal is $\mathfrak{m}=\{x\in R:x^2=0\}$.

$(ii)$ For every $u\in R\setminus F$, the subalgebra $F[u]$ is a field if and only if $u^2\notin F^2=\{x^2:x\in F\}$.
\end{rem}

\begin{rem}\label{prop}
A local commutative algebra is a Frobenius algebra if and only if it has a unique minimal ideal, see \cite[(2.1.3)]{jacob}.
In particular for a finite dimensional totally singular conic algebra, being a Frobenius algebra is a purely ring theoretic property and dose not depend on the base field.
\end{rem}

\begin{lem}\label{frob}
Let $R$ be a finite dimensional totally singular conic algebra over a field $F$.
If $\dim_F R=2^{r_F(R)}$, then $R$ is a Frobenius algebra.
\end{lem}

\begin{proof}
Set $n=r_F(R)$ and write $R=F[u_1,\cdots,u_n]$ for some $u_1,\cdots,u_n\in R$.
The $F$-algebra homomorphisms $f_i:F[u_i]\rightarrow R$ defined by $f_i(u_i)=u_i$, $i=1,\cdots,n$, induce a surjective $F$-algebra homomorphism $f:F[u_1]\otimes\cdots\otimes F[u_n]\rightarrow R$.
By dimension count $f$ is an isomorphism.
We know that single generated algebras (i.e., algebras of the form $F[u]$) and the tensor product of Frobenius algebras are Frobenius (see \cite[(2.1.4)]{jacob} and \cite[(2.1.2)]{jacob}),
hence $R$ is a Frobenius $F$-algebra.
\end{proof}

\begin{rem}\label{rn}
The converse of (\ref{frob}) is not necessarily true.
Here we construct a counter example.
Let $F$ be a field of characteristic $2$ and for $n\geq4$, let $R_n$ be the $n$-dimensional algebra over $F$ with the basis
$\{1,u_1,\cdots,u_{n-1}\}$ subject to the relations
\begin{align}\label{re}
u_i^2=u_1u_i=0&,\quad 1\leq i\leq n-1,\nonumber\\
u_iu_j=u_1&,\quad 2\leq i\neq j\leq n-1.
\end{align}
It is easy to see that the above relations imply that $u_1u_i=u_iu_1$ for every $i$.
It follows that $R_n$ is a totally singular conic algebra with the unique maximal ideal $\mathfrak{m}=Fu_1+\cdots+Fu_{n-1}$.
In particular for every element $x\in R_n$ one can write $x=a+m$, where $a\in F$ and $m\in\mathfrak{m}$.
Set $I:=R_nu_1=Fu_1$.
Then $I$ is a minimal ideal of $R_n$.
Let $x,y\in R_n$ and write $x=a+m$ and $y=b+m'$ where $a,b\in F$ and $m,m'\in\mathfrak{m}$.
By (\ref{re}) we have $mm'\in I$.
So there exist $c\in F$ such that
\begin{align}\label{n1}
xy=ay+bx+cu_1-ab.
\end{align}
Let $r=r_F(R_n)$ and write $R_n=F[v_1,\cdots,v_r]$ for some $v_1,\cdots,v_r\in R_n$.
Set $S=Fv_1+\cdots+Fv_r\subseteq R_n$.
By (\ref{n1}) every monomial in terms of $v_1,\cdots,v_r$ belongs to the subspace $S+Fu_1+F$.
Since these monomials generate $R_n$ as an $F$-algebra and $\dim_FR_n=n$, we obtain $r\geq n-2$.
On the other hand $R_n$ is generated as an $F$-algebra by the elements $u_2,\cdots,u_{n-1}$, so $r_F(R_n)=r=n-2$.

Now suppose that $n$ is even.
We claim that $I$ is the unique minimal ideal of $R_n$.
Let $J\neq\{0\}$ be an ideal of $R_n$ and let $0\neq x\in J$.
As $J\subseteq\mathfrak{m}$ we have
\begin{align}\label{re2}
x=\sum_{i=1}^{n-1}a_iu_i,
\end{align}
 for $a_1,\cdots,a_{n-1}\in F$.
 Set $b_i=(\sum_{j=2}^{n-1}a_j)-a_i$, $i=2,\cdots,n-1$.
Multiplying (\ref{re2}) by $u_i$ we get $b_iu_1\in J$, $i=2,\cdots,n-1$.
If $b_i\neq0$ for some $2\leq i\leq n-1$ then $u_1\in J$ and $I\subseteq J$.
Otherwise $b_2=\cdots=b_{n-1}=0$ which leads to a system of linear equations with respect to $a_2,\cdots,a_{n-1}$.
As $n$ is even (and $\car F=2$) it is easy to see that the only solution of this system is the trivial solution, i.e.,
 $a_2=\cdots=a_{n-1}=0$.
Since $x\neq0$ we obtain $a_1\neq0$, so again $u_1\in J$, i.e., $I\subseteq J$.
So the claim is proved and by (\ref{prop}), $R_n$ is a Frobenius algebra.
For every even integer $n\geq6$ the algebra $R_n$ is a totally singular conic algebra which is Frobenius, but $\dim_F R_n=n\neq2^{n-2}=2^{r_F(R_n)}$.
Also even if $\dim_F R$ is a power of $2$, the converse of (\ref{frob}) is not true; take $n=2^k$, $k\geq3$ and $R=R_n$.
\end{rem}

\begin{rem}\label{tower}
Let $R$ be a finite dimensional totally singular conic algebra over a field $F$ of characteristic $2$ and
let $L\supseteq F$ be a subfield of $R$.
If $R$ is a field, then $r_L(R)+r_F(L)=r_F(R)$.
This fact is an easy consequence of the multiplication formula $[R:F]=[R:L][L:F]$ and the fact that $[K:F]=2^{r_F(K)}$
for every subfield $K\supseteq F$ of $R$.
\end{rem}

\begin{lem}\label{sub}
Let $R$ be a finite dimensional totally singular conic algebra over a field $F$ of characteristic $2$ and let $\mathfrak{m}$ be its unique maximal ideal mentioned in {\rm (\ref{local})}.
Set $r=r_F(R/\mathfrak{m})$ and $n=r_F(R)$.
\begin{itemize}
\item[$(i)$] If $K\supseteq F$ is a maximal subfield of $R$, then the residue field $R/\mathfrak{m}$ and $K$ are isomorphic as $F$-algebras.
    In particular $K$ is unique up to $F$-algebra isomorphism and $[K:F]=\dim_F R/\mathfrak{m}=2^r$.
    Also for every $x\in R$ we have $x^2\in K^2$.
\item[$(ii)$] There exist a maximal subfield $K\supseteq F$ of $R$ and $u_1,\cdots,u_{n-r}\in\mathfrak{m}$ such that $R=K[u_1,\cdots,u_{n-r}]$.

\item[$(iii)$] We have $\ell\ell(R)\leq r_F(R)-r_F(R/\mathfrak{m})+1$.
\end{itemize}
\end{lem}

\begin{proof}
$(i)$ Since $F\subseteq K\subseteq R$, $R$ is a finite dimensional totally singular conic $K$-algebra as well.
Since $K$ is maximal, using (\ref{local} $(ii)$) we have $x^2\in K^2$ for every $x\in R$.
Consider the map $\phi:K\rightarrow R/\mathfrak{m}$ defined by $\phi(x)=x+\mathfrak{m}$.
Clearly $\phi$ is an injective $F$-algebra homomorphism.
We show that $\phi$ is surjective.
Let $x+\mathfrak{m}\in R/\mathfrak{m}$, where $x\in R$.
As $x^2\in K^2$, there exists $y\in K$ such that $x^2=y^2\in K^2$, i.e., $(y+x)^2=0$.
So we have $y+x\in \mathfrak{m}$ which implies that $\phi(y)=y+\mathfrak{m}=x+\mathfrak{m}$.

$(ii)$ By induction on $n$, we first prove that there exist a maximal subfield $K\supseteq F$ of $R$ and $u_1,\cdots,u_{n-r}\in R$ such that $R=K[u_1,\cdots,u_{n-r}]$.
Choose $v_1,\cdots,v_n\in R$ such that $R=F[v_1,\cdots,v_n]$.
If $v_i^2\in F^2$, for every $i=1,\cdots,n$, then for every $u\in R$ we obtain $u^2\in F^2$, so (\ref{local} $(ii)$) implies that every maximal subfield $K\supseteq F$ of $R$ reduces to $F$, i.e., $r=0$ and we are done.
Otherwise (by re-indexing if necessary) we may assume that $v_1^2\notin F^2$.
Then $L:=F[v_1]$ is a quadratic extension of $F$ and $R=L[v_2,\cdots,v_n]$.
By (\ref{prop}), $R$ is a Frobenius $L$-algebra.
We have $r_L(R)=n-1$, also (\ref{tower}) implies that $r_L(R/\mathfrak{m})=r_F(R/\mathfrak{m})-r_F(L)=r-1$. So by induction hypothesis there exist a maximal subfield $K\supseteq L$ of $R$ and $u_1,\cdots,u_{n-r}\in R$
 such that $R=K[u_1,\cdots$ $,u_{n-r}]$.

Since $K$ is maximal we have $u_i^2\in K^2$, $i=1,\cdots,n-r$.
Replacing $u_i$ with $u_i+\alpha_i$ for some $\alpha_i\in K$,
we may assume that $u_i^2=0$, $i=1,\cdots,n-r$.

$(iii)$
By the previous part, there exist a maximal subfield $K\supseteq F$ and $u_1,\cdots,u_{n-r}\in\mathfrak{m}$ such that $R=K[u_1,\cdots,u_{n-r}]$.
Consider arbitrary elements $x_1,\cdots,x_{n-r+1}\in \mathfrak{m}$.
Since $K\cap\mathfrak{m}=\{0\}$, every $x_i$ can be written as
\begin{align*}
x_i&=\sum_{\substack {1\leq l\leq n-r\\ 1\leq i_1<\cdots<i_l\leq n-r}}\alpha_{i_1\cdots i_l} u_{i_1}\cdots u_{i_l},\quad \ \alpha_{i_1\cdots i_l}\in K.
\end{align*}
So every monomial in the expansion of $x_1\cdots x_{n-r+1}$ in terms of $u_1,\cdots,u_{n-r}$ has two identical $u_i$'s.
As $u_i^2=0$ we have $x_1\cdots x_{n-r+1}$ $=0$.
Thus we obtain $\mathfrak{m}^{n-r+1}=0$, i.e., $\ell\ell(R)\leq r_F(R)-r_F(R/\mathfrak{m})+1$.
\end{proof}

The following result shows that (\ref{tower}) is also true for every finite dimensional totally singular conic algebra:
\begin{cor}\label{rf}
Let $R$ be a finite dimensional totally singular conic algebra over a field $F$ of characteristic $2$.
If $L\supseteq F$ is a subfield of $R$ then $r_L(R)+r_F(L)=r_F(R)$.
\end{cor}

\begin{proof}
If $L=F$ (in other words $r_F(L)=0$) the result trivially holds.
So suppose that $r_F(L)\geq 1$.
We obviously have $r_F(R)\leq r_L(R)+r_F(L)$.
So it is enough to show that $r_F(R)\geq r_L(R)+r_F(L)$.
We prove this for the case where $r_F(L)=1$.
The general case follows from induction.
Let $\mathfrak{m}$ be the unique maximal ideal of $R$ mentioned in {\rm (\ref{local})}, $r=r_F(R/\mathfrak{m})$ and $n=r_F(R)$.
By (\ref{sub} $(ii)$) there exist a maximal subfield $K\supseteq F$ of $R$ and $u_1,\cdots,u_{n-r}\in\mathfrak{m}$ such that $R=K[u_1,\cdots$ $,u_{n-r}]$.
Write $L=F[u]$ for some $u\in R$ with $u^2\in F^\times\setminus F^{\times2}$.
Extend $L$ to a maximal subfield $K'$ of $R$.
By (\ref{sub} $(i)$) we have $K'\simeq K$, so there exists $v_1\in K$ such that $v_1^2=u^2\in F^\times\setminus F^{\times2}$.
Set $m:=v_1+u\in R$.
Since $m^2=(v_1+u)^2=0$, we have $m\in \mathfrak{m}$.
As $r_F(K)=r$ and $K$ is a field, by (\ref{tower}) we have $r_{F[v_1]}(K)=r-1$.
So there exist $v_2,\cdots,v_r\in K$ such that $K=F[v_1,\cdots,v_r]$.
Then $R=F[v_{1},\cdots,v_r,u_{1},\cdots,u_{n-r}]$.
Set $S=L[v_{2},\cdots,v_r,u_{1},\cdots,u_{n-r}]$.
We claim that $S=R$ which implies that $r_L(R)\leq n-1=r_F(R)-1$.
It is enough to show that $v_1\in S$.
As $R=K[u_{1},\cdots$ $,u_{n-r}]$, one can write
\begin{align}\label{j}
m&=\sum_{\substack {1\leq l\leq n-r\\ 1\leq i_1<\cdots<i_l\leq n-r}}\alpha_{i_1\cdots i_l} u_{i_1}\cdots u_{i_l},\quad \ \alpha_{i_1\cdots i_l}\in K.
\end{align}
Every $\alpha_{i_1\cdots i_l}\in K=F[v_1,\cdots,v_r]$ can be written as
\[\alpha_{i_1\cdots i_l}=\beta_{i_1\cdots i_l}+\gamma_{i_1\cdots i_l}v_1,\]
where $\beta_{i_1\cdots i_l},\gamma_{i_1\cdots i_l}\in F[v_2,\cdots,v_r]\subseteq S$.
So by (\ref{j}) there exist $s,s'\in S$ such that $m=v_1s+s'$.
As $u_1,\cdots,u_{n-r}\in\mathfrak{m}$ and $\mathfrak{m}$ is an ideal, we obtain $s,s'\in\mathfrak{m}$, so $s,s'\in S\cap\mathfrak{m}$.
We obtain therefore $u=v_1+m=v_1(1+s)+s'$, so $(1+s)u=v_1+(1+s)s'$, i.e., $v_1=(1+s)(u+s')\in S$.
\end{proof}

\begin{rem}\label{gene}
The statement of (\ref{sub} $(ii)$) can be strengthened as follows:
``For every maximal subfield $K\supseteq F$ of $R$ there exist $u_1,\cdots,u_{n-r}\in\mathfrak{m}$ such that
$R=K[u_1,\cdots,u_{n-r}].$''
In fact by (\ref{rf}) we have $r_K(R)=n-r$.
So there exist $u_1,\cdots,u_{n-r}\in R$ such that $R=K[u_{1},\cdots,u_{n-r}]$.
Since $K$ is maximal we have $u_i^2\in K^2$ for $i=1,\cdots,n-r$.
Replacing $u_i$ with $u_i+\alpha_i$ for some $\alpha_i\in K$,
we may assume that $u_1^2=\cdots=u_{n-r}^2=0$.
\end{rem}

\begin{defi}
Let $R$ be a finite dimensional totally singular conic algebra over a field $F$ of characteristic $2$ and let $\mathfrak{m}$ be its unique maximal ideal.
We say that $R$ is a {\it $\rho$-generated} algebra if $r_F(R)=\rho(R)$ where $\rho(R)=\ell\ell(R)+r_F(R/\mathfrak{m})-1$.
\end{defi}

\begin{prop}\label{max}
Let $F$ be a field of characteristic $2$ and let $R$ be a finite dimensional totally singular conic $F$-algebra.
Then $R$ is $\rho$-generated if and only if $\dim_FR=2^{r_F(R)}$.
In particular every $\rho$-generated totally singular conic algebra is a Frobenius algebra.
\end{prop}

\begin{proof}
By (\ref{local}), $R$ is a local commutative algebra.
Let $\mathfrak{m}$ be the unique maximal ideal of $R$, $r=r_F(R/\mathfrak{m})$ and $n=r_F(R)$.

Suppose that $R$ is a $\rho$-generated algebra.
As $r_F(R)=n$ we have $\dim_FR\leq 2^n$.
So it is enough to show that $\dim_FR\geq 2^n$.
Let $K\supseteq F$ be a maximal subfield of $R$ and write $K=F[u_1,\cdots,u_r]$ for some $u_1,\cdots,u_r\in K$.
Since $\ell\ell(R)=n-r+1$ we have $\mathfrak{m}^{n-r}\neq 0$, so there exist $v_1,\cdots,v_{n-r}$ $\in\mathfrak{m}$ such that $v_1\cdots v_{n-r}\neq0$.
We show that
\begin{align}\label{s1}
\dim_K K[v_1,\cdots,v_{n-r}]=2^{n-r},
\end{align}
 which concludes that
\[\dim_FR\geq\dim_FK\cdot\dim_KK[v_1,\cdots,v_{n-r}]=2^n.\]
In order to prove (\ref{s1}), we claim that the set
\[W=\{1\}\cup\{v_{i_1}\cdots v_{i_l}:1\leq i_1<\cdots<i_l\leq n-r,~ 1\leq l\leq n-r\},\]
is linearly independent over $K$.
Suppose that
\begin{align}\label{v}
\sum_{\substack {1\leq l\leq n-r\\ 1\leq i_1<\cdots<i_l\leq n-r}}\alpha_{i_1\cdots i_l} v_{i_1}\cdots v_{i_l}=\alpha,\quad \ \alpha,\alpha_{i_1\cdots i_l}\in K,
\end{align}
where at least one of the above terms is nonzero and the number of nonzero terms is minimal.
Since the left side of (\ref{v}) belongs to $\mathfrak{m}$ we have $\alpha=0$.
 As $v_j^2=0$, $j=1,\cdots,n-r$, multiplying the equality (\ref{v}) by $v_j$ implies that either $v_j$ does not appear in the above sum or appears in all terms.
It follows that the only nonzero term of the left side of (\ref{v}) is a multiple of $v_{j_1}\cdots v_{j_l}$ for some
$1\leq j_1<\cdots<j_l\leq n-r$ and $1\leq l\leq n-r$
which contradicts the assumption $v_1\cdots v_{n-r}\neq0$.
So the claim is proved and $\dim_K K[v_1,\cdots,v_{n-r}]=2^{n-r}$.

Conversely suppose that $\dim_FR=2^n$.
As $\dim_FK=2^r$, we have $\dim_KR=2^{n-r}$.
By (\ref{gene}) there exist $u_1,\cdots,u_{n-r}\in\mathfrak{m}$ such that $R=K[u_1,\cdots,u_{n-r}]$.
It follows that $u_{1}\cdots u_{n-r}\neq0$, i.e., $\mathfrak{m}^{n-r}\neq0$.
So $\ell\ell(R)\geq n-r+1$ and thanks to (\ref{sub} $(iii)$), $R$ is a $\rho$-generated algebra.
The last statement of the result follows from (\ref{frob}).
\end{proof}

\begin{rem}
For $n=2^k$, $k\geq3$, the totally singular conic algebra $R_n$ constructed in (\ref{rn}) is a Frobenius algebra which is not $\rho$-generated, because $\dim_FR_n\neq 2^{r_F(R_n)}$.
So the converse of the second statement of (\ref{max}) does not hold.
\end{rem}

The next result follows from (\ref{max}) and the standard properties of tensor product.
\begin{cor}\label{tensorr}
Let $R$ and $R'$ be two finite dimensional $\rho$-generated totally singular conic algebras over a field $F$ of characteristic $2$.
Then $R\otimes_FR'$ is also a $\rho$-generated totally singular conic $F$-algebra.
\end{cor}

\begin{lem}\label{clif}
Let $(V,q)$ be a quadratic form over a field $F$ of characteristic $2$ and let $f:V\rightarrow C(V)$ be an $F$-linear map such that
 $f(v)\in Z(C(V))$ for every $v\in V$.
 Then the map $f$ can be uniquely extended to an $F$-derivation $\delta:C(V)\rightarrow C(V)$.
\end{lem}

\begin{proof}
Let $T(V)$ be the tensor algebra of $V$ with the canonical map $\bar\ :T(V)\rightarrow C(V)$.
Let $g:T(V)\rightarrow C(V)$ be the linear map induced by $g(1)=0$ and
\begin{align}
g(u_1\otimes\cdots\otimes u_n)&=f(u_1)\overline{u_2}\cdots \overline{u_n}+\overline{u_1}f(u_2)\overline{u_3}\cdots \overline{u_n} \nonumber \\
&\quad +\cdots+\overline{u_1}\cdots \overline{u_{n-1}}f(u_n),
\end{align}
for $u_1,\cdots,u_n\in V$.
For every $w_1,w_2\in T(V)$ we have
\begin{align}\label{deri}
g(w_1\otimes w_2)=\overline{w_1}g(w_2)+g(w_1)\overline{w_2}.
\end{align}
Let $I$ be the ideal of $T(V)$ generated by the elements of the form $v\otimes v-q(v)$ for $v\in V$.
We claim that $g(I)=0$.
For every $w\in T(V)$ and $v\in V$ we have
\begin{align*}
g(w\otimes(v\otimes v-q(v)))&=\overline{w}g(v\otimes v-q(v))+g(w)(\overline{v\otimes v-q(v)})\\
&=\overline{w}g(v\otimes v-q(v))+0=\overline{w}(g(v\otimes v)-g(q(v)))\\
&=\overline{w}g(v\otimes v)=\overline{w}(\overline{v}f(v)+f(v)\overline{v})=0,
\end{align*}
where in the last equality we used the assumption $f(v)\in Z(C(V))$.
Similarly for every $w\in T(V)$ and $v\in V$ we have
$g((v\otimes v-q(v))\otimes w)=0$.
So $g(I)=0$ and $g$ can be factored through $C(V)$, i.e., $g$
induces a map $\delta:C(V)\rightarrow C(V)$.
By (\ref{deri}), for every $w_1,w_2\in C(V)$ we have $\delta(w_1w_2)=w_1\delta(w_2)+\delta(w_1)w_2$, so $\delta$ is a derivation.

The uniqueness of $\delta$ follows from the fact that $V$ generates $C(V)$ as an $F$-algebra.
\end{proof}

\begin{lem}\label{cliff}
Let $R$ be a finite dimensional algebra over a field $F$ of characteristic $2$ and let $n=r_F(R)$.
Then $R$ is a $\rho$-generated totally singular conic $F$-algebra if and only if there exists a totally singular quadratic form $(V,q)$ of dimension $n$ over $F$ such that $R\simeq C(V)$.
In addition $V$ can be chosen as the vector space generated by every generating subset $\{u_1,\cdots,u_n\}$ of $R$ {\rm(}as $F$-algebra{\rm)} with $q(v)=v^2\in F$ for every $v\in V$.
\end{lem}

\begin{proof}
First suppose that $R$ is a $\rho$-generated totally singular conic $F$-algebra.
Write $R=F[u_1,\cdots,u_n]$ for some $u_1,\cdots,u_n\in R$.
Set $\alpha_i=u_i^2\in F$, $i=1,\cdots,n$ and $V=Fu_1+\cdots+Fu_n\subseteq R$.
Define the map $q:V\rightarrow F$ via $q(v)=v^2\in F$.
For every $u,v\in V$ we have $q(u+v)-q(u)-q(v)=(u+v)^2-u^2-v^2=0$.
So $q$ is a totally singular quadratic form.
Consider the inclusion map $i:V\hookrightarrow R$.
The map $i$ is compatible with $q$ and can be extended to an $F$-algebra homomorphism $\phi:C(V)\rightarrow R$.
As $R=F[u_1,\cdots,u_n]$, $\phi$ is surjective.
Also using (\ref{max}), we have $\dim_FR=2^n=\dim_FC(V)$, so $\phi$ is an isomorphism.

Conversely suppose that $R\simeq C(V)$ for a totally singular quadratic form $(V,q)$ over $F$.
Let $\{v_1,\cdots,v_n\}$ be a basis of $V$ over $F$.
As $(V,q)$ is totally singular, $C(V)$ is a totally singular conic $F$-algebra.
Also as an $F$-algebra, $C(V)$ is generated by the set $\{v_1,\cdots,v_n\}$, i.e., $r_F(C(V))\leq n$.
Since $\dim_F C(V)=2^n$ we obtain $n=r_F(C(V))$.
So $\dim_F C(V)=2^{r_F(C(V))}$ and by (\ref{max}), $C(V)$ (and therefore $R$) is $\rho$-generated.
\end{proof}

\begin{cor}\label{new37}
Let $R$ and $R'$ be two finite dimensional $\rho$-generated totally singular conic algebras over a field $F$ of characteristic $2$ with respective maximal subfields $K$ and $K'$.
Then there exists an $F$-algebra isomorphism $R\simeq R'$ if and only if $\dim_FR=\dim_FR'$ and $K\simeq K'$ as $F$-algebras.
\end{cor}

\begin{proof}
First suppose that $\dim_FR=\dim_FR'$ and $K\simeq K'$.
By (\ref{max}) we have $\dim_FR=2^{r_F(R)}$ and $\dim_FR'=2^{r_F(R')}$.
So $\dim_FR=\dim_FR'=2^n$, where $n=r_F(R)=r_F(R')$.
Let $\mathfrak{m}$ and $\mathfrak{m}'$ be respectively the unique maximal ideals of $R$ and $R'$ mentioned in (\ref{local}) and set $r=r_F(R/\mathfrak{m})$.
As $K\simeq K'$, by (\ref{sub} $(i)$) we have $R/\mathfrak{m}\simeq R'/\mathfrak{m}'$ as $F$-algebras.
So $r_F(R'/\mathfrak{m}')=r$.
Also by (\ref{gene}) there exist $u_1,\cdots,u_{n-r}\in\mathfrak{m}$ and $u'_1,\cdots,u'_{n-r}\in\mathfrak{m}'$ such that
$R=K[u_1,\cdots,u_{n-r}]$ and $R'=K'[u'_1,\cdots,u'_{n-r}]$.
As $\dim_FR=2^n$ and $\dim_F K=2^r$, we have $\dim_K R=2^{n-r}$.
So the set $\{u_1,\cdots,u_{n-r}\}$ is linearly independent over $K$.
Set $V=Ku_1+\cdots+Ku_{n-r}$ and $V'=K'u'_1+\cdots+K'u'_{n-r}$.
Define the quadratic forms $q$ on $V$ and $q'$ on $V'$ via $q(v)=v^2$ and $q'(v')={v'}^2$.
Since $u_i\in\mathfrak{m}$ and $u'_i\in\mathfrak{m}'$, $i=1,\cdots,n-r$, $q$ and $q'$ are zero maps.
So $C(V)\simeq C(V')$ as $K$-algebras.
On the other hand by (\ref{cliff}) we have $R\simeq C(V)$ and $R'\simeq C(V')$ as $F$-algebras, so $R\simeq R'$ and we are done.
The converse is trivial.
\end{proof}

\begin{prop}\label{well}
Let $R$ be a finite dimensional totally singular conic algebra over a field $F$ of characteristic $2$ and let $n=r_F(R)$.
Consider a subset $\mathcal{B}:=\{u_1,\cdots,u_n\}\subseteq R$ which generates $R$ as an $F$-algebra.
Then the following statements are equivalent:
\begin{itemize}
\item[$(i)$] The algebra $R$ is $\rho$-generated.
\item[$(ii)$] Every map $f:\mathcal{B}\rightarrow R$ can be uniquely extended to an $F$-derivation $\delta:R\rightarrow R$.
\item[$(iii)$] Every map $f:\mathcal{B}\rightarrow A$ satisfying $f(u_i)^2=u_i^2\in F$, $i=1,\cdots,n$, can be uniquely extended to an $F$-algebra homomorphism $\phi:R\rightarrow A$.
\end{itemize}
\end{prop}

\begin{proof}
Let $V=Fu_1+\cdots+Fu_n$ and let $q:V\rightarrow F$ be the quadratic form defined by $q(v)=v^2$.
By (\ref{cliff}), $R$ is $\rho$-generated if and only if $R\simeq C(V)$.
So the equivalence $(i)\Leftrightarrow (iii)$ follows from the universal property of Clifford algebras and
the implication $(i)\Rightarrow (ii)$ follows from (\ref{clif}).

\noindent$(ii)\Rightarrow (i)$:
Let $K$ be a maximal subfield of $R$ and set $r=r_F(K)$.
Let $\mathfrak{m}$ be the unique maximal ideal of $R$ mentioned in (\ref{local}).
By (\ref{gene}) there exist $v_1,\cdots,v_{n-r}\in\mathfrak{m}$ such that $R=K[v_1,\cdots,v_{n-r}]$.
We claim that $v_{1}\cdots v_{n-r}\neq0$.
If $v_{1}\cdots v_{n-r}=0$,
then there exists a minimal number $l\leq n-r$ and $1\leq i_1<i_2<\cdots<i_l\leq n-r$ such that
\begin{align}\label{min}
v_{i_1}\cdots v_{i_l}=0.
\end{align}
Choose $v_{n-r+1},\cdots,v_n\in R$ such that $K=F[v_{n-r+1},\cdots,v_n]$, so
$\{v_1,\cdots,v_n\}$ generates $R$ as an $F$-algebra.
Let $f:\{v_1,\cdots,v_n\}\rightarrow R$ be the map defined by $f(v_{i_1})=1$ and $f(v_j)=0$ for every $j\neq i_1$.
By the hypothesis, $f$ can be extended to a derivation $\delta$ on $R$.
Note that $l\geq2$, because all $v_i$'s are nonzero.
Applying $\delta$ to (\ref{min}) we obtain $v_{i_2}\cdots v_{i_l}=0$, which contradicts the minimality of $l$.
So the claim is proved and we have $\mathfrak{m}^{n-r}\neq0$.
It follows that $\ell\ell(R)\geq n-r+1$ and thanks to (\ref{sub} $(iii)$), $R$ is a $\rho$-generated algebra.
\end{proof}

\section{Decomposability in terms of Frobenius subalgebras}
\begin{rems}\label{inv}
Let $(Q,\sigma)$ be a quaternion algebra with involution over a field $F$ of characteristic $2$.

\noindent $(i)$
There exists $u\in\sym(Q,\sigma)\setminus F$ such that $u^2\in F^\times$.
In fact if $\sigma$ is of symplectic type and $u\in\sym(Q,\sigma)$, then $u$ has trivial reduced trace (see \cite[pp. 25-26]{knus}),
so $u^2\in F$.
Replacing $u$ with $u+1$ we may assume that $u^2\in F^\times$.
If $\sigma$ is of orthogonal type, $\alt(Q,\sigma)$ is one dimensional and by \cite[(2.8 (2))]{knus} every nonzero element in $\alt(Q,\sigma)$ is invertible.
So for every $0\neq u\in\alt(Q,\sigma)$ we have $u^2=\nrd_Q(u)\in F^\times$, where $\nrd_Q(u)$ is the reduced norm of $u$ in $Q$.

\noindent $(ii)$ If $t$ is the transpose involution, $(Q,\sigma)\simeq(M_2(F),t)$ if and only if $\sigma$ is of orthogonal type and $\disc\sigma$ is trivial, see \cite[(7.4)]{knus}.
\end{rems}
The proof of the following result uses the method of that of
\cite[(3.7)]{pari} and \cite{parimala}\footnote{We are grateful to Professor M.-A. Knus for making \cite{parimala} available to us.}.
The technique was only applied for biquaternion algebras, but the main idea of the general case was present there.

\begin{lem}\label{pa}
Let $(A,\sigma)$ be a central simple algebra with involution over a field $F$ of characteristic $2$.
Suppose that there exists a totally singular conic subalgebra $S\subseteq\sym(A,\sigma)$ such that $C_A(S)=S$ and $\dim_FS=2^r$, where $r=r_F(S)$.
\begin{itemize}
\item[$(i)$] There exists a quaternion $F$-subalgebra $Q\subseteq A$ such that $\sigma(Q)=Q$.
\item[$(ii)$] If $S=F[u_1,\cdots,u_r]$, where $u_1,\cdots,u_r\in S$, then the quaternion subalgebra $Q$ in part {\rm($i$)} can be chosen so that $F[u_2,\cdots,u_r]\subseteq C_A(Q)$.
\item[$(iii)$] If $\sigma$ is of orthogonal type, then the quaternion subalgebra $Q$ in part {\rm($i$)} can be chosen so that $S\cap\alt(Q,\sigma|_{Q})$ has an invertible element.
\end{itemize}
\end{lem}

\begin{proof}
$(i)$
By (\ref{frob}), $S$ is a Frobenius algebra.
Write $S=F[u_1,\cdots,u_r]$, where $u_i\in S$.
Replacing $u_i$ with $u_i+1$ if necessary, which doesn't change $F[u_2,\cdots,u_r]$ and $S$, we may assume that
\begin{align}\label{1}
0\neq u_i^2\in F,\ i=1,\cdots,r.
\end{align}
By (\ref{well}), there exists an $F$-derivation $\delta$ of $S$ induced by $\delta(u_1)=u_1$ and $\delta(u_i)=0$, $i=2,\cdots,r$.
By (\ref{223} $(ii)$), $\delta$ extends to an inner derivation $\delta_\xi$ of $A$ for some $\xi\in A$.
As $\delta^2=\delta$ on $S$, the element $\xi^2+\xi$ commutes with $S$, so
$\xi^2+\xi\in C_A(S)=S$.
Let $\eta=\xi^2\in A$.
We have
\begin{align}\label{2}
\eta^2+\eta=\xi^4+\xi^2=(\xi^2+\xi)^2\in F.
\end{align}
Since $\xi$ and $\xi^2$ induce the same derivation $\delta$ on $S$, we have
\begin{align}\label{v1}
\eta u_1+u_1\eta=\xi u_1+u_1\xi=\delta(u_1)=u_1.
\end{align}
For $x\in S$ the relations
\begin{align*}
\delta_{\sigma(\xi)}(x)=\sigma(\xi)x-x\sigma(\xi)=\sigma(\xi)\sigma(x)-\sigma(x)\sigma(\xi)=\sigma(x\xi-\xi x)=x\xi-\xi x,
\end{align*}
imply that the elements $\xi$ and $\sigma(\xi)$ also induce the same derivation $\delta$ on $S$.
Thus
$\sigma(\xi)=\xi+s$ for some $s\in S$
and
\begin{align}\label{a2}
\sigma(\eta)=\sigma(\xi)^2=(\xi+s)^2=\eta+\delta(s)+s^2,
\end{align}
with $\delta(s)\in S$, since $S$ is stable under $\delta$.
Also
\begin{align}\label{5}
s^2\in F,
\end{align}
since $S$ is totally singular conic algebra.
Now consider two cases.\\
Case $1$. If $\delta(s)\in F$, then (\ref{a2}) implies that $\eta+\sigma(\eta)\in F$.
By (\ref{1}), (\ref{2}) and (\ref{v1}), $u_1$ and $\eta$ generate a quaternion algebra $Q$, which is invariant under $\sigma$, because $\eta+\sigma(\eta)\in F$.\\
Case $2$.
If $\delta(s)\in S \setminus F$, we obtain
\begin{align}\label{6}
\delta(s)=\delta^2(s)=\delta_{\xi^2}(\delta(s))=\eta \delta(s)+\delta(s)\eta.
\end{align}
Then by (\ref{2}), (\ref{5}) and (\ref{6}), $\eta$ and $\delta(s)$ generate a quaternion subalgebra $Q$ of $A$, which is invariant under $\sigma$ thanks to (\ref{a2}).

\noindent$(ii)$ Since $\delta(u_2)=\cdots=\delta(u_r)=0$, in both cases above $F[u_2,\cdots,u_r]$ commutes with $\eta$, thus
$F[u_2,\cdots,$ $u_r]\subseteq C_A(Q)$.

\noindent$(iii)$ Suppose that $\sigma$ is of orthogonal type.
In case $1$ above,
as $\eta+\sigma(\eta)\in F$, we obtain $\eta=\sigma(\eta)$.
Therefore using (\ref{v1}) we get $\eta u_1+\sigma(\eta u_1)=\eta u_1+u_1\eta=u_1$, so $u_1\in S\cap\alt(Q,\sigma|_{Q})$.
Similarly in the case $2$,
 we have $\delta(s)+s^2\in S\cap\alt(Q,\sigma|_{Q})$ by (\ref{a2}).
So in both cases there exists a nonzero element in $S\cap\alt(Q,\sigma|_{Q})$ which is invertible by \cite[(2.8 (2))]{knus}.
\end{proof}

\begin{cor}\label{par}
Let $(A,\sigma)$ be a central simple algebra of degree $2^n$ with involution of orthogonal type over a field $F$ of characteristic $2$.
Suppose that there exists a totally singular conic subalgebra $S\subseteq\sym(A,\sigma)$ such that $C_A(S)=S$ and $\dim_FS=2^{r_F(S)}$.
Then for $i=1,\cdots,n$, there exists a $\sigma$-invariant quaternion algebra $Q_i\subseteq A$ and an invertible element $v_i\in S\cap\alt(Q_i,\sigma|_{Q_i})$ such that $(A,\sigma)\simeq\bigotimes_{i=1}^n(Q_i,\sigma|_{Q_i})$ and $S=F[v_1,\cdots,v_n]$.
In particular $\dim_FS=2^n$.
\end{cor}

\begin{proof}
Set $r=r_F(S)$.
Write $S=F[u_1,\cdots,u_r]$ for some $u_1,\cdots,u_r\in S$ and set $S'=F[u_2,\cdots,u_r]$.
By (\ref{pa} $(i)$) there exists a quaternion subalgebra $Q_1$ of $A$ such that $(A,\sigma)\simeq(Q_1,\sigma|_{Q_1})\otimes(B,\tau)$, where $B=C_A(Q_1)$ and $\tau=\sigma|_B$.
Also by (\ref{pa} $(ii)$), $Q_1$ can be chosen such that $S'\subseteq C_A(Q_1)=B$.
We have
\[F[u_1]\otimes S'\simeq S=C_A(S)\simeq C_{Q_1}(F[u_1])\otimes C_B(S')\simeq F[u_1]\otimes C_B(S'),\]
so $\dim_FC_B(S')=\dim_FS'$.
As $S'$ is commutative we obtain $C_B(S')=S'$.
Since $\dim_FS'=2^{r_F(S')}$, the induction can be applied to $B$ and we obtain $(A,\sigma)\simeq\bigotimes_{i=1}^n(Q_i,\sigma_i)$
where $(Q_i,\sigma_i)$ is a quaternion algebra over $F$ with involution.

By (\ref{pa} $(iii)$) for every $i$ there exists an invertible element $v_i\in S\cap\alt(Q_i,\sigma|_{Q_i})$.
We claim that $S=F[v_1,\cdots,v_n]$.
Set $S''=F[v_1,\cdots,v_n]\subseteq S$.
Then $\dim_FS''=2^n=2^{r_F(S'')}$, so by (\ref{frob}) and (\ref{223} $(i)$), $S''$ is a Frobenius algebra and $C_A(S'')=S''$.
As $S''\subseteq S$ and $C_A(S)=S$, we have $S''=S$.
\end{proof}

\begin{remark}\label{con}
The converse of (\ref{par}) is trivially true.
In fact let $(A,\sigma)=\bigotimes_{i=1}^n(Q_i,\sigma_i)$ be a totally decomposable algebra with involution over a field $F$ of characteristic $2$.
By (\ref{inv}) there exists $v_i\in\sym(Q_i,\sigma_i)$ such that $v_i^2\in F^\times$.
Set $S=F[v_1,\cdots,v_n]\subseteq \sym(A,\sigma)$.
Then $S$ is a totally singular conic $F$-algebra and $\dim_FS=2^n=2^{r_F(S)}$.
By (\ref{frob}) and (\ref{223} $(i)$), we have $C_A(S)=S$.
\end{remark}

\begin{thm}\label{der}
Let $(A,\sigma)$ be a central simple algebra of degree $2^n$ with involution over a field $F$ of characteristic $2$. Then the following statements are equivalent:
\begin{itemize}
\item[$(i)$] $(A,\sigma)$ is totally decomposable.
\item[$(ii)$] There exists a $\rho$-generated totally singular conic $F$-algebra $S\subseteq \sym(A,\sigma)$ such that $C_A(S)=S$.
 \end{itemize}
Furthermore if $\sigma$ is of orthogonal type then for every $F$-algebra $S$ satisfying $(ii)$, we have necessarily $S\subseteq \alt(A,\sigma)+F$.
More precisely
there exist $v_1,\cdots,v_n\in\alt(A,\sigma)\cap A^*$
such that $S=F[v_1,\cdots,v_n]$ and $v_{i_1}\cdots v_{i_l}\in\alt(A,\sigma)$ for every $1\leq i_1<\cdots<i_l\leq n$ and $1\leq l\leq n$.
Thus, if $\sigma$ is of orthogonal type the statement $(ii)$ can be replaced by the following:
\begin{itemize}
\item[(iii)] There exists a $\rho$-generated totally singular conic $F$-algebra $S\subseteq \alt(A,\sigma)+F$ such that $C_A(S)=S$.
 \end{itemize}
Finally if $S$ is any subalgebra of $A$ satisfying the condition $(ii)$ or $(iii)$ then $r_F(S)=n$ and $\dim_FS$ $=\deg_FA=2^n$.
\end{thm}

\begin{proof}
The equivalence $(i)\Leftrightarrow(ii)$ follows from (\ref{max}), (\ref{par}) and (\ref{con}).

If $\sigma$ is of orthogonal type, then by (\ref{par}) for every $i$ there exists an invertible element $v_i\in S\cap\alt(Q_i,\sigma_i)$ such that $S=F[v_1,\cdots,v_n]$.
Choose an element $w_i\in Q_i$ such that $v_i=w_i-\sigma(w_i)$, $i=1,\cdots,n$.
Then $v_{i_1}\cdots v_{i_l}=w_{i_1}v_{i_2}\cdots v_{i_l}-\sigma(w_{i_1}v_{i_2}\cdots v_{i_l})\in\alt(A,\sigma)$ for every $1\leq i_1<\cdots<i_l\leq n$ and $1\leq l\leq n$.

The last statement of the result follows from (\ref{par}).
\end{proof}

\section{A new description of the Pfister invariant}
\begin{defi}
Let $(A,\sigma)$ be a totally decomposable algebra with involution of orthogonal type over a field $F$ of characteristic $2$.
By (\ref{der} $(ii)$) there exists a $\rho$-generated totally singular conic algebra $S\subseteq \sym(A,\sigma)$ such that $C_A(S)=S$.
By (\ref{der} $(iii)$) we have necessarily $S\subseteq\alt(A,\sigma)+F$ and there exist $v_1,\cdots,v_n\in\alt(A,\sigma)\cap A^*$
such that $S=F[v_1,\cdots,v_n]$ and $v_{i_1}\cdots v_{i_l}\in\alt(A,\sigma)$ for $1\leq i_1<\cdots<i_l\leq n$ and $1\leq l\leq n$.
We call the set $\{v_1,\cdots,v_n\}$, a {\it set of alternating generators} of $S$.
\end{defi}

\begin{defi}\label{bil}
Let $(A,\sigma)$ be a totally decomposable algebra of degree $2^n$ with involution of orthogonal type over a field $F$ of characteristic $2$ and let $S\subseteq \sym(A,\sigma)$ be a $\rho$-generated totally singular conic algebra such that $C_A(S)=S$.
Define a map $\mathfrak{s}:S\times S\rightarrow F$ as follows:
for every $v,w\in S$, as $vw\in S\subseteq\alt(A,\sigma)+F$, there exists a unique $\alpha\in F$ such that $vw+\alpha\in \alt(A,\sigma)$.
Set $\mathfrak{s}(v,w)=\alpha$.
It is easy to see that $\mathfrak{s}$ is a symmetric bilinear form on $S$.
Note that for every $v\in S$, as $v^2\in F$ and $F\cap\alt(A,\sigma)=\{0\}$, we obtain
\begin{align}\label{star}
\mathfrak{s}(v,v)=v^2\in F.
\end{align}
Furthermore for every $u,v,w\in S$ and $\alpha\in F$, we have $u(vw)+\alpha\in \alt(A,\sigma)$ if and only if $(uv)w+\alpha\in \alt(A,\sigma)$, so $\mathfrak{s}(u,vw)=\mathfrak{s}(uv,w)$, i.e., $\mathfrak{s}$ is an {\it associative} bilinear form.
We also have the orthogonal decomposition $S=F\perp S_0$ with respect to $\mathfrak{s}$, where $S_0=S\cap\alt(A,\sigma)$.
\end{defi}

\begin{rem}\label{ax}
Let $(A,\sigma)$ be a totally decomposable algebra of degree $2^n$ with involution of orthogonal type over a field $F$ of characteristic $2$ and let $S\subseteq \sym(A,\sigma)$ be a $\rho$-generated totally singular conic algebra such that $C_A(S)=S$.
If $\{v_1,\cdots,v_n\}$ is a set of alternating generators of $S$ with $v_i^2=\alpha_i\in F^\times$ and $\mathfrak{s}$ is the bilinear form defined in $(\ref{bil})$, then
$\mathfrak{s}\simeq\lla\alpha_1,\cdots,\alpha_n\rra$; in particular $\mathfrak{s}$ is nondegenerate.
\end{rem}

\begin{defi}
Let $(A,\sigma)\simeq \bigotimes_{i=1}^n(Q_i,\sigma_i)$ be a totally decomposable algebra with involution of orthogonal type over a field $F$ of characteristic $2$ and let $\alpha_i\in F^{\times}$, $i=1,\cdots,n$, be a representative of the class $\disc\sigma_i\in F^{\times}/F^{\times2}$.
In \cite{dolphin3} it is shown that the bilinear $n$-fold Pfister form $\mathfrak{Pf}(A,\sigma):=\lla\alpha_1,\cdots,\alpha_n\rra$ is independent of the decomposition of $(A,\sigma)$.
As in \cite{dolphin3} we call this form the {\it Pfister invariant} of $(A,\sigma)$.
\end{defi}

The following result gives another description of $\mathfrak{Pf}(A,\sigma)$:
\begin{lem}\label{lp}
Let $(A,\sigma)$ be a totally decomposable algebra with involution of orthogonal type over a field $F$ of characteristic $2$ and let $S\subseteq \sym(A,\sigma)$ be a $\rho$-generated totally singular conic algebra such that $C_A(S)=S$.
Then the bilinear form $\mathfrak{s}$ on $S$ defined in $(\ref{bil})$ is isometric to $\mathfrak{Pf}(A,\sigma)$.
\end{lem}

\begin{proof}
By (\ref{par}) there exists a $\sigma$-invariant quaternion algebra $Q_i\subseteq A$, $i=1,\cdots$ $,n$ with $(A,\sigma)\simeq\bigotimes_{i=1}^n(Q_i,\sigma|_{Q_i})$ and an invertible element $v_i\in\alt(Q_i,\sigma|_{Q_i})$ such that $S=F[v_1,\cdots,v_n]$ and $v_i^2=\alpha_i\in F^\times$, $i=1,\cdots,n$.
Then $\{v_1,\cdots,v_n\}$ is set of alternating generators of $S$ and $\mathfrak{s}\simeq\lla\alpha_1,\cdots,\alpha_n\rra$.
We also have $\disc\sigma_i$$=\nrd_{Q_i}(v_i)F^{\times2}=\alpha_i F^{\times2}\in F^\times/F^{\times2}$,
so $\mathfrak{Pf}(A,\sigma)\simeq\lla\alpha_1,\cdots,\alpha_n\rra\simeq\mathfrak{s}$.
\end{proof}

Let $\alpha_1,\cdots,\alpha_n,\beta_1,\cdots,\beta_n\in F^\times$.
The bilinear Pfister forms $\lla\alpha_1,\cdots,\alpha_n\rra$ and $\lla\beta_1,\cdots,\beta_n\rra$
are said to be {\it simply P-equivalent} if either $n=1$ and $\alpha_1F^{\times2}=\beta_1 F^{\times2}$ or
$n\geq2$ and there exist $1\leq i<j\leq n$ such that
$\lla\alpha_i,\alpha_j\rra\simeq\lla\beta_i,\beta_j\rra$ and $\alpha_k=\beta_k$ for $k\neq i,j$.
We say that two bilinear Pfister forms ${\mathfrak b}$ and ${\mathfrak c}$ are {\it chain P-equivalent}, if there exist
bilinear Pfister forms ${\mathfrak b}_0,\cdots,{\mathfrak b}_m$ such that ${\mathfrak b}={\mathfrak b}_0$, ${\mathfrak c}={\mathfrak b}_m$ and for every $i=0,\cdots,m-1$,
${\mathfrak b}_i$ and ${\mathfrak b}_{i+1}$ are simply P-equivalent.

\begin{prop}\label{gen}
Let $(A,\sigma)$ be a totally decomposable algebra of degree $2^n$ with involution of orthogonal type over a field $F$ of characteristic $2$ and let $S\subseteq \sym(A,\sigma)$ be a $\rho$-generated totally singular conic algebra such that $C_A(S)=S$.
Let $\mathfrak{s}$ be the bilinear form on $S$ defined in \rm{(\ref{bil})}.
If $\mathfrak{s}\simeq\lla\alpha_1,\cdots,\alpha_n\rra$ for some $\alpha_1,\cdots,\alpha_n\in F^\times$,
then there exists a set of alternating generators $\{u_1,\cdots,u_n\}$ of $S$ such that $u_i^2=\alpha_i$, $i=1,\cdots,n$.
\end{prop}

\begin{proof}
First suppose that $n=1$ and let $\{v_1\}$ is a set of alternating generators of $S$ with $v_1^2=\beta\in F^\times$.
Since $A$ is a quaternion algebra, we have
$\disc\sigma=\nrd_{A}(v_1)=\beta F^{\times2}\in F^\times/F^{\times2}$ and the result follows from \cite[(7.4)]{knus}.

Now suppose that $n=2$.
Set $S_0=S\cap\alt(A,\sigma)=F^\perp$, $\mathfrak{s}_0=\mathfrak{s}|_{S_0\times S_0}$ and $\mathfrak{b}=\lla\alpha_1,\alpha_2\rra$.
As $\mathfrak{s}\simeq\mathfrak{b}$, by \cite[p. 16]{arason} the pure subforms of these forms are also isometric.
So $\mathfrak{s}_0\simeq\langle\alpha_1,\alpha_2,\alpha_1\alpha_2\rangle$.
Let $V$ be the underlying vector space of $\mathfrak{b}_0:=\langle\alpha_1,\alpha_2,\alpha_1\alpha_2\rangle$ with a respective orthogonal basis $\{v_1,v_2,v_3\}$.
Let $f:(V,\mathfrak{b}_0)\simeq(S_0,\mathfrak{s}_0)$ be an isometry and set $u_i=f(v_i)\in S_0$, $i=1,2$.
We have $u_i^2=f(v_i)^2=v_i^2=\alpha_i$, $i=1,2$.
We also have $\mathfrak{s}(1,u_1u_2)=\mathfrak{s}(u_1,u_2)=\mathfrak{b}(v_1,v_2)=0$, so $u_1u_2\in F^\perp=S_0\subseteq\alt(A,\sigma)$.
Thus
$\{u_1,u_2\}$ is the desired set of alternating generators of $S$.

Finally suppose that $n\geq3$.
Let $(A,\sigma)\simeq\bigotimes_{i=1}^n(Q_i,\sigma_i)$ be a decomposition of $(A,\sigma)$ and let $v_i\in\alt(Q_i,\sigma_i)$ with $v_i^2=\beta_i\in F^\times$, $i=1,\cdots,n$.
Then $\{v_1,\cdots,v_n\}$ is set of alternating generators of $S$ and by (\ref{ax}) we have ${\mathfrak s}\simeq\lla\beta_1,\cdots,\beta_n\rra$, so \[\lla\alpha_1,\cdots,\alpha_n\rra\simeq\lla\beta_1,\cdots,\beta_n\rra.\]
By \cite[(A. 1)]{arason} there exist
bilinear $n$-fold Pfister forms ${\mathfrak b}_0,\cdots,{\mathfrak b}_m$ such that ${\mathfrak b}_0=\lla\beta_1,\cdots,\beta_n\rra$, ${\mathfrak b}_m=\lla\alpha_1,\cdots,\alpha_n\rra$ and for every $i=0,\cdots,m-1$,
${\mathfrak b}_i$ is simply P-equivalent to ${\mathfrak b}_{i+1}$.
In order to prove the result, using induction on $m$, it is enough to consider the case where $\lla\alpha_1,\cdots,\alpha_n\rra$ and $\lla\beta_1,\cdots,\beta_n\rra$ are simply P-equivalent.
By re-indexing if necessary we may assume that $\lla\alpha_1,\alpha_2\rra\simeq\lla\beta_1,\beta_2\rra$ and $\alpha_i=\beta_i$ for $i\geq3$.
Set $(B,\tau)=(Q_1,\sigma_1)\otimes(Q_2,\sigma_2)$ and $S'=F[v_1,v_2]$.
Then $S'\subseteq \sym(B,\tau)$ is a $\rho$-generated totally singular conic algebra and $C_B(S')=S'$.
As proved in the case $n=2$, there exists a set of alternating generators $\{u_1,u_2\}$ of $S'$
 such that $u_i^2=\alpha_i$, $i=1,2$.
We have $S=S'\otimes F[v_3,\cdots,v_n]\simeq F[u_1,u_2]\otimes F[v_3,\cdots,v_n]$.
So $\{u_1,u_2,v_3,\cdots,v_n\}$ is the desired set of alternating generators of $S$.
\end{proof}

\begin{prop}\label{tr}
Let $(A,\sigma)$ be a totally decomposable algebra of degree $2^n$ with involution of orthogonal type over a field $F$ of characteristic $2$ and let $S\subseteq \sym(A,\sigma)$ be a $\rho$-generated totally singular conic algebra such that $C_A(S)=S$.
Then the following statements are equivalent:
$(i)$ $(A,\sigma)\simeq(M_{2^n}(F),t)$.
$(ii)$ $\mathfrak {Pf}(A,\sigma)\simeq\lla1,\cdots,1\rra$.
$(iii)$ $u^2\in F^2$ for every $u\in S$.
$(iv)$ Every maximal subfield of $S$ containing $F$ reduces to $F$.
\end{prop}

\begin{proof}
The implication $(i)\Rightarrow (ii)$ follows from the fact that the transpose involution
 in characteristic $2$ has trivial discriminant, see \cite[p. 82]{knus}.

$(ii)\Rightarrow (i):$
Let $(A,\sigma)\simeq\bigotimes_{i=1}^n(Q_i,\sigma_i)$ be a decomposition of $(A,\sigma)$ to quaternion algebras with involution over $F$.
Since $\car F=2$ a sum of squares in $F$ is again a square, so $D_F(\mathfrak {Pf}(A,\sigma))=D_F(\lla1,\cdots,1\rra)=F^2$.
It follows that for every $i$, $\disc\sigma_i$ is trivial.
So by (\ref{inv}) we have $(Q_i,\sigma_i)\simeq(M_2(F),t)$ which implies that $(A,\sigma)\simeq(M_{2^n}(F),t)$.
The equivalence $(ii)\Leftrightarrow (iii)$ follows from (\ref{ax}), (\ref{lp}) and (\ref{gen}) and $(iii)\Leftrightarrow (iv)$ follows from (\ref{local}).
\end{proof}

The proof of the following result is left to the reader.
\begin{lem}\label{sk}
Let $(A,\sigma)$ be a totally decomposable algebra with involution of orthogonal type over a field $F$ of characteristic $2$ and let $S\subseteq \sym(A,\sigma)$ be a $\rho$-generated totally singular conic $F$-algebra such that $C_A(S)=S$.
If $K/F$ is a field extension, then $S_K\subseteq \sym(A_K,\sigma_K)$ is a $\rho$-generated totally singular conic $K$-algebra and $C_{A_K}(S_K)=S_K$.
\end{lem}

\begin{lem}\label{spl}
Let $(A,\sigma)$ be a totally decomposable algebra of degree $2^n$ with involution of orthogonal type over a field $F$ of characteristic $2$ and let $S\subseteq \sym(A,\sigma)$ be a $\rho$-generated totally singular conic algebra such that $C_A(S)=S$.
If $K\supseteq F$ is a maximal subfield of $S$, then $(A_K,\sigma_K)\simeq(M_{2^n}(K),t)$.
In particular $K$ is a splitting field of $A$.
\end{lem}

\begin{proof}
By (\ref{sk}), $S_K\subseteq \sym(A_K,\sigma_K)$ is a $\rho$-generated totally singular conic algebra and $C_{A_K}(S_K)=S_K$.
As $K$ is a maximal subfield of $S$, by (\ref{local}) we have $u^2\in K^2$ for every $u\in S$.
This, together with $K^2\subseteq S^2\subseteq F$ implies that $x^2\in K^2$ for every $x\in S_K$.
So by (\ref{tr}) we have $(A_K,\sigma_K)\simeq(M_{2^n}(K),t)$.
\end{proof}

Let $(A,\sigma)$ be a totally decomposable algebra with involution of orthogonal type over a field $F$ of characteristic $2$.
By \cite[(7.2)]{dolphin3}, (\ref{lp}) and (\ref{gen}) all $\rho$-generated totally singular conic subalgebras $S$ of $(A,\sigma)$ with $S\subseteq \sym(A,\sigma)$ and $C_A(S)=S$ are isomorphic as $F$-algebras.
Here, we give a proof of this fact which is independent from \cite{dolphin3}:
\begin{lem}\label{field}
Let $(A,\sigma)$ be a totally decomposable algebra of degree $2^n$ with involution of orthogonal type over a field $F$ of characteristic $2$ and let $S$ and $S'$ be two $\rho$-generated totally singular conic subalgebras of $(A,\sigma)$ such that $S,S'\subseteq \sym(A,\sigma)$, $C_A(S)=S$ and $C_A(S')=S'$.
\begin{itemize}
\item[$(i)$]
If $K\supseteq F$ and $K'\supseteq F$ are respectively maximal subfields of $S$ and $S'$, then $K\simeq K'$ as $F$-algebras.
\item[$(ii)$]
We have $S\simeq S'$ as $F$-algebras.
\end{itemize}
\end{lem}

\begin{proof}
By (\ref{spl}) we have $(A_K,\sigma_K)\simeq(M_{2^n}(K),t)$.
Also according to (\ref{sk}), $S'_K$ is a $\rho$-generated totally singular conic subalgebra of $(A_K,\sigma_K)$ with $S'_K\subseteq \sym(A_K,\sigma_K)$ and $C_{A_K}(S'_K)=S'_K$.
So we have $u^2\in K^2$ for every $u\in S'_K$ by (\ref{tr}).
In particular if $K'=F[v'_1,\cdots,v'_r]$, where $r=r_F(K')$ and $v'_1,\cdots,v'_r\in K'\subseteq S'\subseteq S'_K$,
then ${v'_i}^2\in K^2$, $i=1,\cdots,r$.
So for every $i$ there exists $v_i\in K$ such that $v_i^2={v'_i}^2\in K^2$.
By (\ref{well}) the linear map $f:K'\rightarrow K$ induced by $f(v'_i)=v_i$ is an $F$-algebra homomorphism.
As $K'$ is a field, $f$ is a monomorphism.
Similarly there exists an $F$-algebra monomorphism $K\hookrightarrow K'$.
So $\dim_FK=\dim_FK'$ and $K\simeq K'$ as $F$-algebras.
This proves $(i)$.
The statement $(ii)$ follows from $(i)$ and (\ref{new37}).
\end{proof}

\begin{notation}
Let $(A,\sigma)$ be a totally decomposable algebra with involution of orthogonal type over a field $F$ of characteristic $2$.
In view of (\ref{field} $(ii)$) there exists, up to isomorphism, a unique $\rho$-generated totally singular conic algebra $S\subseteq \sym(A,\sigma)$ such that $C_A(S)=S$.
We denote this algebra by $\Phi(A,\sigma)$.
Note that if $\mathfrak{s}$ is the bilinear form on $\Phi(A,\sigma)$ defined in $(\ref{bil})$, then $\mathfrak{s}\simeq\mathfrak{Pf}(A,\sigma)$ by (\ref{lp}).
\end{notation}

\begin{rem}\label{phiq}
Let $(Q,\sigma)$ be a quaternion algebra with involution over a field $F$ of characteristic $2$.
Then $\Phi(Q,\sigma)=F+\alt(Q,\sigma)$.
In fact by (\ref{der}) we have $\Phi(Q,\sigma)\subseteq F+\alt(Q,\sigma)$ and $\dim\Phi(Q,\sigma)=2$.
As $\alt(Q,\sigma)$ is one dimensional, we obtain $\Phi(Q,\sigma)=F+\alt(Q,\sigma)$.
\end{rem}

\begin{cor}\label{tensor}
Let $(A,\sigma)$ and $(B,\tau)$ be two totally decomposable algebras with involution of orthogonal type over a field $F$
of characteristic $2$.
Then we have $\Phi(A\otimes B,\sigma\otimes\tau)\simeq\Phi(A,\sigma)\otimes\Phi(B,\tau)$.
\end{cor}

\begin{proof}
$(i)$ Set $S=\Phi(A,\sigma)\otimes\Phi(B,\tau)$.
Then $S\subseteq\sym(A\otimes B,\sigma\otimes\tau)$ is a totally singular conic $F$-algebra and $C_{A\otimes B}(S)=S$.
Also by (\ref{tensorr}), $S$ is a $\rho$-generated algebra.
So by (\ref{field} $(ii)$), we have $\Phi(A\otimes B,\sigma\otimes\tau)\simeq S$.
The part $(ii)$ follows from (\ref{sk}) and (\ref{field} $(ii)$).
\end{proof}

\section{Some characterizing properties of the Pfister invariant}
\begin{lem}\label{q0} {\rm(Compare \cite[(2.4)]{dherte})}
Let $K/F$ be a finite extension of fields of characteristic $2$ and
let $(Q,\sigma)$ be a quaternion algebra with involution of orthogonal type over $K$.
If $u^2\in F$ for every $u\in\Phi(Q,\sigma)$, then there exists a quaternion $F$-subalgebra $Q_0\subseteq Q$ such that $\sigma(Q_0)=Q_0$, i.e., $(Q,\sigma)\simeq_K(Q_0,\sigma|_{Q_0})\otimes(K,\id)$.
\end{lem}

\begin{proof}
The idea of the proof is similar to the proof of (\ref{pa}).
Set $S=\Phi(Q,\sigma)$.
By (\ref{phiq}) we have $S=K+\alt(Q,\sigma)=K[u]$, where $0\neq u\in\alt(Q,\sigma)$.
Let $\delta$ be the $K$-derivation of $S$ induced by $\delta(u)=u$.
By (\ref{223} $(ii)$), $\delta$ extends to an inner derivation $\delta_\xi$ of $Q$ for some $\xi\in Q$.
Let $\eta=\xi^2\in Q$.
As $\xi^2+\xi\in C_Q(S)=S$, we obtain $(\xi^2+\xi)^2\in F$ and $(\xi+\xi^2)u=u(\xi+\xi^2)$, thus
\[\eta^2+\eta=\xi^4+\xi^2=(\xi^2+\xi)^2\in F,\quad \eta u+u\eta=\xi u+u\xi=\delta(u)=u.\]
So the $F$-algebra generated by $u$ and $\eta$ is a quaternion algebra.
As $\alt(Q,\sigma)=Ku$, we have $\eta+\sigma(\eta)=\alpha u$ for some $\alpha\in K$.
If $\alpha\in F$, then the quaternion $F$-algebra $Q_0$ generated by $u$ and $\eta$ is invariant under $\sigma$.
Otherwise $\alpha\neq0$ and the quaternion $F$-algebra $Q_0$ generated by $\alpha u$ and $\eta$ is invariant under $\sigma$.
\end{proof}

\begin{cor}\label{q}
Let $K/F$ be a finite extension of fields of characteristic $2$.
Let $(A,\sigma)$ be a totally decomposable algebra with involution of ortho\-gonal type over $K$.
If $u^2\in F$ for every $u\in\Phi(A,\sigma)$, then there exists a central simple $F$-algebra $B\subseteq A$ such that $(A,\sigma)\simeq_K(B,\sigma|_B)\otimes(K,\id)$.
\end{cor}

\begin{proof}
Let $(A,\sigma)\simeq_K\bigotimes_{i=1}^n(Q_i,\sigma_i)$ be a decomposition of $(A,\sigma)$ into quaternion $K$-algebras with involution and choose an invertible element $u_i\in\alt(Q_i,\sigma_i)$, $i=1,\cdots,n$.
Then we have $\Phi(A,\sigma)\simeq K[u_1,\cdots,u_n]$, so $u_i^2\in F^\times$ for $i=1,\cdots,n$.
It follows from (\ref{tensor}) that $u^2\in F$ for every $u\in\Phi(Q_i,\sigma_i)$.
So by (\ref{q0}), $(A,\sigma)\simeq_K\bigotimes_{i=1}^{n}(Q'_i,\sigma|_{Q'_i})\otimes(K,\id)$,
where $Q'_i$ is a quaternion $F$-subalgebra of $Q_i$.
\end{proof}

\begin{lem}\label{k}
Let $(A,\sigma)$ be a totally decomposable algebra of degree $2^n$ with involution of orthogonal type over a field $F$ of characteristic $2$.
Consider an element $u\in \Phi(A,\sigma)$ with $u^2\in F^\times\setminus F^{\times2}$.
Set $B=C_A(u)$ and $K=F[u]$.
\begin{itemize}
\item[$(i)$] The pair $(B,\sigma|_B)$ is a totally decomposable algebra with involution of ortho\-gonal type over $K$ and $\Phi(B,\sigma|_B)\simeq\Phi(A,\sigma)$ as $K$-algebras.
\item[$(ii)$] There exists a quaternion algebra $Q\subseteq A$ containing $u$ such that $\sigma(Q)=Q$.
\end{itemize}
\end{lem}

\begin{proof}
$(i)$ Since $\sigma(u)=u$, $\sigma|_B$ is of the first kind.
We also have $1\notin\alt(A,\sigma)$ which implies that $1\notin\alt(B,\sigma|_B)$, so $\sigma|_B$ is of orthogonal type.
Set $S=\Phi(A,\sigma)$.
By (\ref{rf}) we have $r_K(S)=r_F(S)-1=n-1$.
So $\dim_KS=2^{r_K(S)}$ and using (\ref{max}), $S$ is a Frobenius $\rho$-generated $K$-algebra.
As $\dim_KS=\deg_KB$, by (\ref{223} $(i)$) we have $C_B(S)=S$.
By (\ref{der}) and (\ref{field} $(ii)$) $(B,\sigma|_B)$ is totally decomposable and $\Phi(B,\sigma|_B)\simeq S$.

$(ii)$ As $u^2\in F$ for every $u\in\Phi(B,\sigma|_B)\simeq S$, by (\ref{q}) there exists a central simple $F$-algebra $B_0\subseteq B$ such that
$(B,\sigma|_B)\simeq_K(B_0,\sigma|_{B_0})\otimes(K,\id)$.
Then $Q=C_A(B_0)\subseteq A$ is a quaternion algebra containing $u$ which is invariant under~$\sigma$.
\end{proof}

\begin{lem}\label{k2}
Let $(A,\sigma)$ be a totally decomposable algebra with involution of orthogonal type over a field $F$ of characteristic $2$.
Let $\{u_1,\cdots,u_n\}$ be a set of alternating generators of $\Phi(A,\sigma)$ and set $\alpha_i=u_i^2\in F^\times$, $i=1,\cdots,n$.
Suppose that $\alpha_n\notin F^2$ and set $B=C_A(u_n)$ and $K=F[u_n]$.
Then $(B,\sigma|_B)$ is a totally decomposable algebra with involution of orthogonal type over $K$ and
$\{u_1,\cdots,u_{n-1}\}$ is a set of alternating generators of $\Phi(B,\sigma|_B)$.
In particular $\mathfrak{Pf}(B,\sigma|_B)\simeq\lla\alpha_1,\cdots,\alpha_{n-1}\rra_K$.
\end{lem}

\begin{proof}
By (\ref{k} $(i)$), $(B,\sigma|_B)$ is a totally decomposable algebra with involution of orthogonal type over $K$ and we have an isomorphism of $K$-algebras $\Phi(B,\sigma|_B)\simeq\Phi(A,\sigma)$.
Let $1\leq l\leq n-1$ and $1\leq i_1<\cdots<i_l\leq n-1$.
We claim that $u_{i_1}\cdots u_{i_l}\in\alt(B,\sigma|_B)$.
By (\ref{der} $(iii)$) we have $\Phi(B,\sigma|_B)\subseteq \alt(B,\sigma|_B)+K$, so there exists $\lambda\in K$ such that $w:=u_{i_1}\cdots u_{i_l}+\lambda\in\alt(B,\sigma|_B)$.
Write $\lambda=\alpha+\beta u_n$ for some $\alpha,\beta\in F$.
Then
\[w=u_{i_1}\cdots u_{i_l}+\alpha+\beta u_n\in\alt(B,\sigma|_B)\subseteq\alt(A,\sigma).\]
As $\sigma$ is of orthogonal type, we have $\alpha=0$.
As $u_n\in K=Z(B)$ we have $u_n w\in\alt(B,\sigma|_B)$.
On the other hand $u_n w=u_nu_{i_1}\cdots u_{i_l}+\alpha_n\beta\in\alt(B,\sigma|_B)\subseteq\alt(A,\sigma)$.
Again, as $\sigma$ is of orthogonal type, we have $\beta=0$.
So $u_{i_1}\cdots u_{i_l}\in\alt(B,\sigma|_B)$ and the claim is proved, i.e.,
$\{u_1,\cdots,u_{n-1}\}$ is a set of alternating generators of $\Phi(B,\sigma|_B)$.
\end{proof}

In \cite{dolphin3} it was shown that if $(A,\sigma)\simeq(A',\sigma')$ then $A\simeq A'$ and $\mathfrak {Pf}(A,\sigma)\simeq\mathfrak {Pf}(A',\sigma')$.
It was also asked whether the converse is also true (see \cite[(7.4)]{dolphin3}).
The following result shows that this question has an affirmative answer, i.e., totally decomposable algebras with involution of orthogonal type can be classified, up to conjugation, by their Pfister invariant.

\begin{thm}\label{phi}
Let $(A,\sigma)$ and $(A',\sigma')$ be two totally decomposable algebras with involution of orthogonal type over a field $F$ of characteristic $2$.
If $A\simeq A'$ and $\mathfrak {Pf}(A,\sigma)\simeq\mathfrak {Pf}(A',\sigma')$, then $(A,\sigma)\simeq(A',\sigma')$.
\end{thm}

\begin{proof}
Let $\{u_1,\cdots,u_n\}$ be a set of alternating generators of $\Phi(A,\sigma)$ with $u_i^2=\alpha_i\in F^\times$, so that
$\mathfrak{Pf}(A,\sigma)\simeq\lla\alpha_1,\cdots,\alpha_n\rra$.
By (\ref{lp}) and (\ref{gen}) there exists a set of alternating generators $\{u'_1,\cdots,u'_n\}$ of $\Phi(A',\sigma')$ such that ${u'_i}^2=\alpha_i$, $i=1,\cdots,n$.

We use induction on $n$.
If $n=1$, we have $\disc\sigma=\disc\sigma'=\nrd_A(u_1)F^{\times2}=\alpha_1 F^{\times2}$, so the result follows from \cite[(7.4)]{knus}.
So suppose that $n>1$.
If $\alpha_i\in F^{\times2}$ for every $i=1,\cdots,n$, then using (\ref{tr}) we obtain $(A,\sigma)\simeq(A',\sigma')\simeq(M_{2^n}(F),t)$ and we are done.
So (by re-indexing if necessary) we may assume that $\alpha_n\in F^\times\setminus F^{\times2}$.
Set $B=C_A(u_n)$, $K=F[u_n]$, $B'=C_{A'}(u'_n)$ and $K'=F[u'_n]$.
As $K\simeq K'=F(\sqrt\alpha_n)$, we may consider $B'$ as a central simple algebra over $K$.
By (\ref{k} $(i)$) and (\ref{k2}), $(B,\sigma|_B)$ and $(B',\sigma'|_{B'})$ are totally decomposable algebras with involution of orthogonal type over $K$ and $\mathfrak{Pf}(B,\sigma|_B)\simeq\mathfrak{Pf}(B',\sigma'|_{B'})\simeq\lla\alpha_1,\cdots,\alpha_{n-1}\rra_K$.
Since $A\simeq_F A'$, we obtain $B\simeq_K B'$, so by induction hypothesis there exists an isomorphism of $K$-algebras with involution
\begin{align*}
f:(B,\sigma|_B)\simeq_K(B',\sigma'|_{B'}).
\end{align*}
By (\ref{k} $(i)$) we have a $K$-algebra isomorphism $\Phi(B,\sigma|_B)\simeq\Phi(A,\sigma)$.
So we get $u^2\in F$ for every $u\in\Phi(B,\sigma|_B)$.
By (\ref{q}) there exists a central simple $F$-subalgebra $B_0\subseteq B$ such that
$(B,\sigma|_B)\simeq_K(B_0,\sigma|_{B_0})\otimes(K,\id)$.
Set $B_0'=f(B_0)\subseteq B'$, hence $(B_0,\sigma|_{B_0})\simeq_F(B_0',\sigma'|_{B_0'})$.
Then $B_0'$ is a $\sigma'$-invariant central simple subalgebra of $B'$ and $(B',\sigma'|_{B'})\simeq_K(B_0',\sigma'|_{B_0'})\otimes(K,\id)$.
Set $Q:=C_A(B_0)\supseteq C_A(B)=C_A(C_A(u_n))$, hence $u_n\in Q$.
Similarly set $Q':=C_{A'}(B'_0)=C_{A'}(f(B_0))\supseteq C_{A'}(f(B))=C_{A'}(B')=C_{A'}(C_{A'}(u'_n))$, hence $u'_n\in Q'$.
As $\deg_FQ=\deg_FQ'=2$, $Q$ and $Q'$ are quaternion algebras over $F$.
We also have $(A,\sigma)\simeq_F(Q,\sigma|_{Q})\otimes(B_0,\sigma|_{B_0})$ and $(A',\sigma')\simeq_F(Q',\sigma'|_{Q'})\otimes(B_0',\sigma'|_{B_0'})$.
Since $B_0\simeq_FB_0'$ and $A\simeq_FA'$, we obtain
\begin{align*}
\textstyle Q'\simeq_F C_{A'}(B_0')\simeq_F C_{A}(B_0)\simeq_F Q.
\end{align*}
Also as $u_n\in\alt(A,\sigma)\cap Q$, by \cite[(3.5)]{mahmoudi2} we have $u_n\in\alt(Q,\sigma|_Q)$.
It follows that $\disc\sigma|_{Q}=\nrd_{Q}(u_n)F^{\times2}=\alpha_n F^{\times2}$ and similarly $\disc\sigma'|_{Q'}=\alpha_n F^{\times2}$, hence $(Q,\sigma|_Q)\simeq_F(Q',\sigma'|_{Q'})$ by \cite[(7.4)]{knus}.
Using this isomorphism we obtain
\begin{align*}
(A,\sigma)\simeq_F(Q,\sigma|_{Q})\otimes(B_0,\sigma|_{B_0})\simeq_F(Q',\sigma'|_{Q'})\otimes(B_0',\sigma'|_{B_0'})\simeq(A',\sigma').\hspace{1.7cm}\qedhere
\end{align*}
\end{proof}
A bilinear space $(V,\mathfrak{b})$ over a field $F$ is called {\it metabolic} if there exists a subspace $W$ of $V$ such that $\dim W=\frac{1}{2}\dim V$ and $\mathfrak{b}|_{W\times W}=0$.
An $F$-algebra with involution $(A,\sigma)$ is called {\it metabolic} if there exists an idempotent $e\in A$ such that $\sigma(e)e = 0$ and $\dim_FeA = \frac{1}{2}\dim_FA$.
A bilinear form is metabolic if and only if its adjoint involution is metabolic, see \cite[(4.8)]{dolphin}.

As an application we complement a characterization of totally decomposable algebras with metabolic involution given in \cite{dolphin3}.

\begin{thm}{\rm(\cite[(7.5)]{dolphin3})}\label{pfi}
Let $(A,\sigma)$ be a totally decomposable algebra of degree $2^n$ with involution of orthogonal type over a field $F$ of characteristic $2$.
Then the following statements are equivalent:
\begin{itemize}
\item[$(i)$] $(A,\sigma)$ is metabolic.
\item[$(ii)$] $\mathfrak{Pf}(A,\sigma)$ is metabolic.
\item[$(iii)$] $\Phi(A,\sigma)$ is not a field.
\item[$(iv)$] There exists a central simple algebra with involution of orthogonal type $(B,\tau)$ over $F$
such that $(A,\sigma)\simeq(M_2(F),t)\otimes(B,\tau)$.
\end{itemize}
\end{thm}

\begin{proof}
The equivalence of $(i)$ and $(ii)$ was shown in \cite[(7.5)]{dolphin3}.

$(ii)\Rightarrow (iii)$: If $\mathfrak{Pf}(A,\sigma)$ is metabolic, then as $\mathfrak{Pf}(A,\sigma)\simeq\mathfrak{s}$, by the relation (\ref{star}) given in (\ref{bil}) there exists a nonzero $x\in\Phi(A,\sigma)$ such that $x^2=0$
and we obtain $(iii)$.

$(iii)\Rightarrow (iv)$:
We use induction on $n$.
If $n=1$, the result follows from (\ref{tr}) and (\ref{phiq}).
So suppose that $n>1$.
Let $\mathfrak{m}$ be the unique maximal ideal of $\Phi(A,\sigma)$ and let $K\supseteq F$ be a maximal subfield of $\Phi(A,\sigma)$.
If $K=F$ the result again follows from (\ref{tr}); so suppose that $K\neq F$.
Write $K=F[u_1,\cdots,u_r]$, where $r=r_F(K)$ and $u_1,\cdots,u_r\in K$.
Since $\Phi(A,\sigma)$ is not a field we have $\mathfrak{m}\neq\{0\}$.
So using (\ref{gene}) one can find $u_{r+1},\cdots,u_n\in\mathfrak{m}$ such that
$\Phi(A,\sigma)=K[u_{r+1},\cdots,u_n]$.
It follows that $\Phi(A,\sigma)=F[u_1,\cdots,u_n]$ with $u_1\in K\setminus F$ and $u_n\in\mathfrak{m}$.
By (\ref{pa} $(ii)$) there exists a $\sigma$-invariant quaternion subalgebra $Q$ of $A$ such that $F[u_2,\cdots,u_n]\subseteq C_A(Q)$.
 Set $B=C_A(Q)$.
We have $\dim_F F[u_2,\cdots,u_n]=\deg_FB=2^{n-1}$, so by (\ref{max}), $F[u_2,\cdots,u_n]$ is a $\rho$-generated $F$-algebra.
Also it is easy to see that $C_B(F[u_2,\cdots,u_n])=F[u_2,\cdots,u_n]$.
As $F[u_2,\cdots,u_n]\subseteq \sym(B,\sigma|_B)$, by (\ref{der}), $(B,\sigma|_B)$ is a totally decomposable algebra with involution of orthogonal type over $F$.
By (\ref{field} $(ii)$) we have
$\Phi(B,\sigma|_B)\simeq F[u_2,\cdots,u_n]$.
 As $u_n\in F[u_2,\cdots,u_n]\cap\mathfrak{m}$, $\Phi(B,\sigma|_B)$ is not a field.
So the result follows from induction hypothesis.

$(iv)\Rightarrow (i)$: Let
\begin{align*}
e=\left(\begin{array}{cc}0 & 1 \\0 & 1\end{array}\right)\in M_2(F).
\end{align*}
Then $e$ is a metabolic idempotent for $(M_2(F),t)$.
So $(M_2(F),t)$ is metabolic which implies that $(A,\sigma)\simeq(M_2(F),t)\otimes(B,\tau)$ is also metabolic.
\end{proof}

\footnotesize
\noindent{\sc M. G. Mahmoudi, {\tt
    mmahmoudi@sharif.ir}, \ \
  A.-H. Nokhodkar, {\tt
    anokhodkar@yahoo.com}\\   Department of Mathematical Sciences, Sharif University of Technology, P. O. Box 11155-9415, Tehran, Iran.}

\end{document}